\documentclass[10pt,a4paper]{amsart}
\usepackage[toc,page]{appendix}
\usepackage{hyperref}
\usepackage{amsthm}
\usepackage{mathrsfs}
\usepackage{amsmath}
\usepackage{amssymb}
\usepackage[all]{xy}
\usepackage{latexsym}
\usepackage{ bbold }
\usepackage{graphicx}
\usepackage{blindtext}

\mathchardef\dashmod="2D

\usepackage{datetime}
\usepackage{pgf,tikz}
\usetikzlibrary{arrows}

\usepackage[latin1]{inputenc}
\usepackage[T1]{fontenc}
\usepackage{xcolor}

\renewcommand{\ker}{\mathrm{ker}}

\newcommand{\ext}{\mathrm{Ext}}

\renewcommand{\mod}{\mathbf{Mod}}

\newcommand{\st}{\mathbf{St}}
\newcommand{\un}{\mathbb{S}}

\newcommand{\End}{\mathrm{End}}

\newcommand{\Pic}{\mathrm{Pic}}
\newcommand{\can}{\iota}

\newcommand{\Z}{\mathbb{Z}}
\newcommand{\F}{\mathbb{F}}

\newcommand{\A}{\mathcal{A}}

\renewcommand{\hom}{{\sf hom}}

\renewcommand{\L}{\mathcal{L}}
\newcommand{\zero}{{\bf 0}}

\newcommand{\sur}{/ \hspace{-0.07cm}/}

\theoremstyle{definition}
\newtheorem{de}{Definition}[section]

\newtheorem*{conv*}{Convention} 

\theoremstyle{plain}
\newtheorem{thm}[de]{Theorem}
\newtheorem{lemma}[de]{Lemma}
\newtheorem{pro}[de]{Proposition}

\newtheorem*{thm*}{Theorem}
\newtheorem*{lemma*}{Lemma}
\newtheorem*{pro*}{Proposition}
\newtheorem*{cor*}{Corollary}
\newtheorem*{prob*}{Problem}
\newtheorem*{clm*}{Claim}

\theoremstyle{remark}
\newtheorem{rk}[de]{Remark}

\title{The stable Picard group of $\A(2)$}
\author{Prasit Bhattacharya, Nicolas Ricka}
\address{Department of Mathematics, University of Notre Dame \\ 106 Hayes-Healy Hall,\\ \\
Notre Dame, IN 46556, USA}
 \address{Department of Mathematics, Wayne State University \\
 Detroit, MI 48202}
\email{prasbhat@indiana.edu}
\email{nicolas.ricka@wayne.edu}

\begin{document}
\maketitle
\begin{abstract}
Using a form of descent in the stable category of $\A(2)$-modules, we show that there are no exotic elements in the stable Picard group of $\A(2)$, \textit{i.e.} that the stable Picard group of $\A(2)$ is free on $2$ generators.
\end{abstract}

\section*{Acknowledgments}
Authors would like to thank Bob Bruner for some fruitful conversations.

\begin{conv*} Through out this paper, $\F$ will denote the field with two elements. Every algebraic structure is implicitly over the base field $\F$, and tensor products are taken over $\F$. The Hopf algebras under consideration in this paper are connected, cocommutative finite dimensional graded Hopf algebras, unless explicitly specified otherwise. 
\end{conv*} 
\section{Introduction}

Let $A$ be a Hopf algebra. The \emph{Picard group} of $\st(A)$, denoted by $\Pic(A)$ is the group of stably $\otimes$-invertible $A$-modules,
\[ \Pic(A):= \lbrace M \in \st(A): \text{$\exists$ $N$ such that $M \otimes N = \un$} \rbrace, \]
where $\un$ is the unit of the symmetric monoidal category $(\st(A),\otimes,\un)$. When $B \subset A$ be a Hopf subalgebra, the forgetful functor $U:\st(A) \rightarrow \st(B)$ being monoidal, it induces a group homomorphism $\Pic(U) : \Pic(A) \rightarrow \Pic(B).$ Define the \emph{relative Picard group} $\Pic(A,B)$ as the kernel $$\Pic(A,B) :=\ker (\Pic(U)).$$ 

Some elements are always in the Picard group of a Hopf algebra. Explicitly, there is a morphism of groups (see \eqref{eq:canonicalmap}):
\begin{equation*} 
\can : \Z \oplus \Z \rightarrow \Pic(A).
\end{equation*}
The interesting part of the Picard group consist in the the elements in the cokernel of $\can$. These are called \emph{exotic elements.}
In this paper, we are interested in the determination of the Picard group of $\A(2)$, the Hopf subalgebra of $\A$ generated by $Sq^1, Sq^2$ and $Sq^4$. This problem is very natural, as the continuation of the study presented in \cite{AP76}.
Let $\A(1)$ denote the Hopf subalgebra of the Steenrod algebra $\A$ generated by $Sq^1$ and $Sq^2$. Questions regarding Picard groups of hopf algebra started with \cite{AP76}, where the determination of the stable Picard group of the Hopf subalgebra $\A(1)$ is used to show the uniqueness of the infinite loop space structure on the classifying space of the infinite orthogonal group (see \textit{loc cit} for the definition of $\A(1)$). The connective real $K$-theory $\mathit{ko}$ is also related to $\A(1)$ as $H^*(\mathit{ko}, \F) \cong \A \sur \A(1)$ (see \cite{AP76}). Here, the result is quite surprising: the group homomorphism 
$$\can : \Z \oplus \Z \rightarrow \Pic(\A(1))$$ is not surjective, and there is exactly one exotic element, called the Joker (see \cite{AM74,BrOssa}). This $\A(1)$-module is pictured in figure \ref{fig:joker}. 
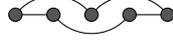
\begin{figure}[h]
\definecolor{qqqqff}{rgb}{0.3333333333333333,0.3333333333333333,0.3333333333333333}
\begin{tikzpicture}[line cap=round,line join=round,>=triangle 45,x=0.5cm,y=0.5cm]
\clip(5.82942940857475,5.953777759762861) rectangle (10.537605207779377,8.02281283886911);
\draw (6.,7.)-- (7.,7.);
\draw (9.,7.)-- (10.,7.);
\draw [shift={(7.,6.25)}] plot[domain=0.6435011087932844:2.498091544796509,variable=\t]({1.*1.25*cos(\t r)+0.*1.25*sin(\t r)},{0.*1.25*cos(\t r)+1.*1.25*sin(\t r)});
\draw [shift={(9.,6.25)}] plot[domain=0.6435011087932844:2.498091544796509,variable=\t]({1.*1.25*cos(\t r)+0.*1.25*sin(\t r)},{0.*1.25*cos(\t r)+1.*1.25*sin(\t r)});
\draw [shift={(8.,7.75)}] plot[domain=3.7850937623830774:5.639684198386302,variable=\t]({1.*1.25*cos(\t r)+0.*1.25*sin(\t r)},{0.*1.25*cos(\t r)+1.*1.25*sin(\t r)});
\begin{scriptsize}
\draw [fill=qqqqff] (6.,7.) circle (2.5pt);
\draw [fill=qqqqff] (7.,7.) circle (2.5pt);
\draw [fill=qqqqff] (8.,7.) circle (2.5pt);
\draw [fill=qqqqff] (9.,7.) circle (2.5pt);
\draw [fill=qqqqff] (10.,7.) circle (2.5pt);
\end{scriptsize}
\end{tikzpicture}
\caption{The joker. Each dot represents a copy of $\F$. Straight lines represent the action of $Sq^1$ and curved ones represent the action of $Sq^2$.} \label{fig:joker}
\end{figure}

Studying $\Pic(\A(2))$ is of topological interest for two reasons. Firstly, $\A(2)$ shares the same relationship with the spectrum of topological modular forms as the relationship between $\A(1)$ and connective real $K$-theory. Secondly, the Joker plays a role in the determination of the Picard group of the $K(1)$-local stable homotopy category \cite{HMS94}.

The main result of this paper shows there are no `surprises' for the stable Picard group of $\A(2)$, more precisely:
\begin{thm*} \ref{thm:picachu}
The morphism of groups
\begin{equation*}
\can : \Z \oplus \Z \rightarrow \Pic(\A(2)),
\end{equation*}
which sends $(n,m)$ to $\un^{n,m}$, is an isomorphism.
\end{thm*}

The key idea in our approach is that we consider a chain of inclusions of Hopf subalgebras of $\A(2)$ (which we  define in section~\ref{sec:setup}),
\[D(2) \subset C(2) \subset \A(2)\]
of Hopf subalgebras of $\A(2)$. Our starting point is to compute $\Pic(D(2))$ by hand and show that (see Proposition~\ref{pro:picD2}) $\Pic(D(2))$ does not have any exotic elements.  The heuristic reason for considering $D(2)$ is that it is very close to being an exterior algebra (see Remark~\ref{rk:notext}) and $\Pic(E)$ for an exterior Hopf algebra is known to be free of exotic elements (see Proposition~\ref{pro:picexterior}). Then, by definition of the relative Picard group, we get a three-stage filtration of the stable Picard group of $\A(2)$.
\begin{equation*}
\xymatrix{
\Pic(\A(2)) \ar[r] &  \Pic(C(2)) \ar[r] &  \Pic(D(2)) \\
\Pic(\A(2),C(2)) \ar@{^(->}[u] &  \Pic(C(2),D(2)) \ar@{^(->}[u] &  \Pic(D(2)), \ar@{=}[u]
}.
\end{equation*}
where the vertical maps are the inclusion of the kernel of the next horizontal map. Next we use methods of homotopical descent along restriction functors as described in \cite{Ric4}. In Corollary~\ref{lem:false}, we give a criteria for general $B \subset A$ under which the relative Picard group $\Pic(A,B)$ is trivial. It follows that $\Pic(C(2),D(2))$ and $\Pic(\A(2),C(2))$ are trivial. Our main result, Theorem~\ref{thm:picachu}, is a straightforward consequence of the above observations.

In light of Corollary~\ref{lem:false}, it seems that the exotic element in the stable Picard group of $\A(1)$ (namely the Joker) is just a low-dimensional anomaly and authors expect the stable Picard group of $\A(n)$ not to have any exotic elements whenever $n \geq 2$\footnote{Eric Wofsey made similar conclusions in his unpublished work.}. Authors hope to address the case when $n>2$ in future.

\tableofcontents
\section{Set up} \label{sec:setup}

Let $A$ be a finite dimensional graded Hopf algebra. Let $\mod(A)$ be the category of bounded below graded $A$-modules. For a graded $A$-module $M$, we denote by $M_n$ the set of elements in degree $n$.  The diagonal $\Delta : A \rightarrow A \otimes A$ gives $\mod(A)$ the structure of a symmetric monoidal closed category. Explicitly, if $M,N$ are $A$-modules, the action of $A$ on $M \otimes N$ is defined by
\begin{eqnarray*}
A \otimes M \otimes N & \stackrel{\Delta \otimes M \otimes N}{\longrightarrow} & A \otimes A \otimes M \otimes N \\
& \stackrel{\simeq}{\rightarrow} & A \otimes M \otimes A \otimes N \\
& \rightarrow & M \otimes N,
\end{eqnarray*}
and the unit $\un$ for this monoidal structure is the $A$-module $\F$, concentrated in degree zero.

Recall from \cite[Theorem 12.5, Proposition 12.8]{Mar83} that the category $\mod(A)$ has enough injective and projective modules, and that the class of injective modules, projective modules, and free modules coincide. Let $\st(A)$ (see \cite[Example 2.4.(v)]{SS03} for an efficient and detailed definition) denote the stable category of graded $A$-modules. Objects of $\st(A)$ are the objects of $\mod(A)$, but the stable morphisms between two $A$-modules $M$ and $N$ is
 \begin{equation} \label{eq:stablemaps}
[M,N]^{A} := \frac{\mod(A)(M,N)}{<\lbrace f :M \rightarrow P \rightarrow N : P \text{ is a projective $A$-module} \rbrace >}. 
\end{equation}
In particular, free (respectively injective, projective, since these notion coincide) $A$-modules  becomes equivalent to $\zero$ in this category. It turns out (see \cite{Mar83}) that the structure of a closed monoidal category on $\mod(A)$ passes to $\st(A)$ . We will denote (abusively) the monoidal product in $\st(A)$ by $\otimes$, the unit by $\un_A$ and the hom-bifunctor by $F_A(-,-)$.

Additionally, the process of `killing the free modules' gives $\st(A)$ the structure of a triangulated category (see \cite[Section 9.6]{HPS}). We recall here the definition of the suspension functor, for completeness.
\begin{de} Choose a minimal projective resolution
\begin{equation*} 
 \dots \rightarrow P^{i+1} \rightarrow P^{i} \rightarrow \dots \rightarrow P^1 \rightarrow P^0 \rightarrow \un_A
\end{equation*}
of $\un_A$ as an $A$-module. For $n\geq 0$, let the $A$-module $\ker(P^{n+1} \rightarrow P^n)$ be denoted by $\un_A^{n,0}$, and the linear dual of $\un_A^{n,0}$ by $\un_A^{-n,0}$. We denote by $\un_A^{n,m}$ the $A$-module $\un_A^{n,0}[-m]$, where $[-m]$ is the shift in internal grading. Explicitly $(\un_A^{n,m})_k = (\un_A^{n,0})_{k+m}$.
 \end{de}
We can now define the suspension functor
\begin{equation*}
\Sigma^{n,m} : \st(A) \rightarrow \st(A)
\end{equation*}
which sends  $M \mapsto \un^{n,m}_A \otimes M$.

\begin{rk} 
Another choice of projective resolution of $\un_A$ would result in another definition of $\un^{n,m}_A$. However, since the resulting objects would be isomorphic to $\un^{n,m}_A$, the choice we made is harmless.
\end{rk}

\begin{de} \label{de:red}
We say that an $A$-module $M$ is reduced if it does not contain any $A$-free summand.
\end{de}
Let $M$ be an $A$-module. As observed in \cite{BrOssa}, one can construct a reduced $A$-module $M^{red}$, which is stably isomorphic to $M$.

To follow the common notations in algebraic topology, we make the convention 
\begin{equation*} \label{eqn:stableconvention} 
[M,N]_{s,t}^A := [\Sigma^{s,t}M,N]^A.
\end{equation*}
\begin{rk}
By \cite[Section 9.6]{HPS}, the suspension functor which is part of the triangulated structure is $\Sigma^{1,0}$. In particular, for $A$-modules $M$ and $N$, and $s \geq 0$, the bigraded extension groups are isomorphic to  
\begin{equation} \label{de:grading}
[M,N]_{-s,t}^A = [\Sigma^{-s,t}M,N]^A \cong \ext^{s,t}_A(M,N) . 
\end{equation} 
In particular when $s=0$, $\ext_A^{0,t}(M,N) = [M,N]_{0,t}^A$ is the set of stable maps, as described in \eqref{eq:stablemaps}. Let $\hom_A(M,N)$ denote the space of homomorphisms between $M$ and $N$. When $s \geq 1$, the vector space 
$$[M,N]_{s,t}^A \cong \pi_s(\hom_A(\Sigma^{0,t}M,N)).$$
The reader is referred to \cite[Proposition 14.1.8]{Mar83} for the explicit comparison.
\end{rk}

We now turn to the definition of our main object of interest: the Picard and relative Picard groups.
\begin{de}
Let $\Pic(A)$ be the group of stably $\otimes$-invertible $A$-modules,
\[ \Pic(A):= \lbrace M \in \st(A): \text{$\exists$ $N$ such that $M \otimes N = \un_A$} \rbrace. \] This group is called the \emph{Picard group} of $\st(A)$, and its elements are called \emph{Picard elements}.
\end{de}
\begin{rk} \label{rk:picardcompact}
Note that Picard elements share a lot of properties with the unit object (see \cite[Proposition 2.1.3]{MS14}). In particular, Picard elements have a finite dimensional representative. This justifies our restriction to bounded below modules. 
\end{rk}

When $B \subset A$ be a Hopf subalgebra, the forgetful functor $U:\st(A) \rightarrow \st(B)$ is monoidal. This induces a group homomorphism $\Pic(U) : \Pic(A) \rightarrow \Pic(B).$ 

\begin{de}
Let 
$$\Pic(A,B) :=\ker (\Pic(U)).$$
This group is called the \emph{relative Picard group} associated to the inclusion $B \subset A$.
\end{de}

Note that there is always a family of $\otimes$-invertible modules. Indeed, for all $(m,n) \in \Z \oplus \Z$, the $A$-module $\un^{n,m}_A$ has inverse $\un^{-n,-m}_A$. This defines a group homomorphism
\begin{equation} \label{eq:canonicalmap} 
\can : \Z \oplus \Z \rightarrow \Pic(A).
\end{equation}
\begin{de}
The elements in the cokernel of $\can$ are called \emph{exotic elements.}
\end{de}

Let $\A(2)$ be the Hopf subalgebra of the modulo $2$ Steenrod algebra $\A$ generated by the Steenrod squares $Sq^1,Sq^2,Sq^4$. As an algebra, $\A(2)$ has the following presentation
\begin{equation*}
\A(2) \cong \frac{\F[Sq^1,Sq^2,Sq^4]}{ \begin{pmatrix} Sq^1Sq^1, \\Sq^2Sq^2 + Sq^1 Sq^2 Sq^1, \\ Sq^1 Sq^4 +Sq^4 Sq^1 + Sq^2 Sq^1 Sq^2, \\ Sq^4 Sq^4 + Sq^2 Sq^4 Sq^2 + Sq^4 Sq^2 Sq^2 \end{pmatrix}}.
\end{equation*}

The coalgebra structure is given by the Cartan formulas. This Hopf algebra is dual to 
\begin{equation*}
\A(2)^{\ast} \cong \frac{\F[\xi_1,\xi_2,\xi_3]}{(\xi_1^8, \xi_2^4, \xi_3^2)},
\end{equation*}
together with the diagonal
\begin{eqnarray*}
\Delta(\xi_1) & = & \xi_1 \otimes 1 + 1 \otimes \xi_1, \\
\Delta(\xi_2) & = & \xi_2 \otimes 1 + \xi_1^2 \otimes \xi_1 + 1 \otimes \xi_2, \\
\Delta(\xi_3) & = & \xi_3 \otimes 1 + \xi_2^2 \otimes \xi_1 + \xi_1^4 \otimes \xi_2 + 1 \otimes \xi_3. 
\end{eqnarray*}

By Palmieri's work \cite[Theorem 1.3]{Pal97}, a stable $\A(2)$-module $M$ is $\otimes$-invertible if and only if, for all quasi-elementary Hopf subalgebras $E$ of $\A(2)$, the restriction of $M$ to $E$ is. Moreover, all the elementary sub-Hopf algebras of $\A(2)$ are in fact exterior Hopf subalgebras by \cite[Section 2.1.1]{Pal01}. Finally \textit{loc cit} recalls the classification the elementary Hopf subalgebras of $\A$ (the result is originally in \cite{Mar83}). We deduce the following result:
\begin{pro}
The maximal elementary Hopf subalgebras of $\A(2)$ are
\begin{eqnarray*}
E_1 & \cong & E(Q_0,Q_1,Q_2), \\
E_2 & \cong & E(Q_1, P_2^1, Q_2),
\end{eqnarray*}
where 
\begin{eqnarray*}
Q_0 & = & Sq^1, \\
Q_{i+1} & = & Q_i Sq^{2^{i+1}} +Sq^{2^{i+1}} Q_i,\\
P_2^1 & = & Sq^2 Sq^4 + Sq^4 Sq^2.
\end{eqnarray*}
\end{pro}

 \begin{de} \label{de:D2}
Let $D(2)$ be the Hopf subalgebra of $\A(2)$ generated by $Q_0,Q_1,Q_2$, and $P_2^1$. 
\end{de}
\begin{de}
Let $C(2)$ be the Hopf subalgebra of $\A(2)$ generated by $Sq^1$, $Sq^2$, $Q_1$, $P_2^1$ and $Q_2$.
\end{de}
The dual of these Hopf algebras have an easy presentation:
\begin{equation*} 
D(2)^{\ast} \cong \frac{\F[\xi_1,\xi_2,\xi_3]}{(\xi_1^2, \xi_2^4, \xi_3^2)},
\end{equation*}  
\begin{equation*}
C(2)^{\ast} \cong \frac{\F[\xi_1,\xi_2,\xi_3]}{(\xi_1^4, \xi_2^4, \xi_3^2)},
\end{equation*}
as quotient Hopf algebras of $\A(2)^{\ast}$.
\begin{rk} \label{rk:notext} Note that every generator $D(2)$ is either contained in $E_1$ or in $E_2$. However, $D(2)$ is not an exterior algebra as the commutator
\[ [Q_0, P_2^1] = Q_0P_2^1 + P_2^1Q_0 = Q_2\]
is nonzero. In fact, 
\[ D(2) = \frac{\F[Q_0, Q_1, P_2^1, Q_2] }{\langle Q_0^0, Q_1^2, Q_2^2,  [Q_0,P_2^1]= Q_2, (P_2^1)^2 \rangle}\]
is a presentation of $D(2)$.
\end{rk}

\section{An elementary case: the Picard group of $D(2)$}

In this section we compute the stable Picard group of $D(2)$. An immediate consequence of the study of exterior algebras in \cite{AM74} is that the Picard group of exterior Hopf algebras does not contain any exotic element (we give another argument here in Proposition \ref{pro:picexterior}, which relies on the analysis of \cite{CT05}). The Hopf algebra $D(2)$ being close to be an exterior Hopf algebra, the reader should expect that its Picard group does not contain exotic elements. Proposition \ref{pro:picD2} shows that it is indeed the case.
\begin{pro} \label{pro:picD2}
The group homomorphism
\begin{equation*}
\can : \Z \oplus \Z \rightarrow \Pic(D(2)),
\end{equation*}
is an isomorphism.
\end{pro}

The proof is direct, and only relies on the analysis of the Picard group of elementary algebras.
\begin{pro} \label{pro:picexterior}
Let $n \geq 2$ and $E = E(x_1,x_2, \hdots x_n)$ be the exterior Hopf algebra generated by the elements $x_i$.
Then, $\can$ is an isomorphism.
\end{pro}
\begin{proof}
Let $M$ be a $\otimes$-invertible module. Equivalently, the morphism
\begin{equation*}
\hom(M,M) \rightarrow \un_E
\end{equation*}
is an isomorphism. The latter assertion depends only on the underlying ungraded module over the exterior \emph{algebra} $E$. But the exterior algebra is isomorphic to the group algebra of an elementary abelian $2$-group. By the classification of $\otimes$-invertibles modules in the group algebra case, due to Carlson and Thevenaz in \cite{CT05}, $M$ is stably isomorphic to $\un^{n,m}_E$. The result follows.
\end{proof}

\begin{proof}[Proof of Proposition \ref{pro:picD2}]
Let $[M]$ be an isomorphism class of stably $\otimes$-invertible modules. By Definition \ref{de:red} and subsequent construction, we can assume without loss of generality that $M$ is a reduced module (see Definition~\ref{de:red}). Moreover, by Remark \ref{rk:picardcompact}), $M$ is finite dimensional in this case.

Let $M_1$ and $M_2$ and $N$ be the restrictions of $M$ to $E_1$, $E_2$ and $E(Q_1,Q_2)$ respectively. There is a commutative diagram of abelian groups such that 
\begin{equation*} \label{piccom}
\xymatrix{ \Pic(D(2)) \ar[r] \ar[d] & \Pic(E_1) \ar[d] && [M] \ar@{|->}[r] \ar@{|->}[d]& [M_1]\ar@{|->}[d]\\
\Pic(E_2) \ar[r] & \Pic(E(Q_1,Q_2)) && [M_2] \ar@{|->}[r] & [N].}
\end{equation*} 
By Proposition \ref{pro:picexterior}, $\Pic(E_1)$, $\Pic(E_2)$ and $\Pic(E(Q_1,Q_2))$ are isomorphic to $\mathbb{Z} \times \mathbb{Z}$, and therefore the left vertical map and the bottom horizontal map in \eqref{piccom} are isomorphisms as shown below:  
\begin{equation*}
\xymatrix{ \Pic(D(2)) \ar[r] \ar[d] & \Pic(E_1) \ar[d]^{\cong} \\
\Pic(E_2) \ar[r]^{\cong} & \Z \oplus \Z.}
\end{equation*}
Consequently, without loss of generality, we can assume that both $M_1$ and $M_2$ are stably isomorphic to $\un_{E_1}$ and $\un_{E_2}$ respectively (if not, replace $M$ by $\Sigma^{n,m}M$ for suitable $m$ and $n$). Let $x \in M$ be an element of smallest degree (this makes sense since the modules are bounded by assumption).

Suppose that $Q_0,Q_1,Q_2$ and $P_2^1$ acts trivially on $x$, then the inclusion
\begin{equation*}
x\F \rightarrow M
\end{equation*}
is split. Since this inclusion induces an isomorphism in Margolis Homology (see \cite{Mar83}, it is an isomorphism, showing the result. 

Suppose now that one of the operations $Q_0,Q_1,Q_2$ and $P_2^1$ acts non-trivially on $x$. Let us assume this operation belongs to $E_1$ (a similar argument gives the result when it belongs to $E_2$). As $M_1$ is stably isomorphic to $\un_{E_1}$, $x$ must then belong to a free $E_1$-submodule. But $x$ cannot be in the target of any operation for degree reasons (it is in minimal degree and the operation have a positive degree). Thus there is an $E_1$-submodule $E_1x$ in $M_1$. In particular, an operation in $E_2$ acts non trivially on $Q_0x$. As $M_2$ is stably isomorphic to $\un_{E_2}$, $Q_0x$ must then belong to a free $E_2$-submodule. But again, $Q_0x$ cannot be in the target of any operation in $E_2$ for degree reasons, giving  $E_2 \cdot Q_0x \subset M_2$.
Consequently, $x$ generates a $D(2)$-free submodule, as $Q_2P_2^1Q_1Q_0x \neq 0$. Now, $D(2)$ is an injective submodule (recall that injective, projective and free modules are the same notion over a connective finite dimensional Hopf algebra), so $D(2)$ splits off. This is in contradiction with the hypothesis that $M$ is reduced.

\end{proof}

\section{Two descents to $\Pic(\A(2))$}

We will first review the algebraic descent spectral sequence, henceforth will be abbreviated to Alg-DSS, which is our main computational tool for our main result. This spectral sequence was developed in \cite{Ric4} and is inspired from the descent spectral sequence which appeared in \cite{MS14}. This section recollects the constructions of \textit{loc cit} and reformulate some of the results in our particular case. The key point of this section is Corollary \ref{lemma:obstructionline}, which gives a explicit condition under which the relative Picard groups are trivial. 

Let  $B \subset A$ be a conormal Hopf subalgebra such that algebra $(A\sur B)^{\ast}$ is exterior on one element $\tau$ in degree $|\tau| \geq 1$. This is exactly the situation we will encounter later, first when $A= C(2)$ and $B=D(2)$, and, second when $A= \A(2)$ and $B= C(2)$. Given an $A$-module $M_A$, let $\End_A(M_A)$ be the space $\hom_A(M_A,M_A)$, i.e. space of homomorphisms from $M_A$ to itself. In particular, $\pi_i(\End_A(M_A)) \cong [M_A,M_A]^A_{i,0}$. In \cite{Ric4}, the author considers the moduli space of stable $A$-modules, over a fixed $B$-module $M_B$. This is a topological space $\L_A(M_B)$ whose connected components are in one-to-one correspondence with the set of stable equivalence classes of $A$-modules $M_A$ such that $U M_A = M_B$. Let $M_A$ be the base point for $\L_A(M_B)$. Then one should note that 
\[ \Omega\L_A(M_B) \cong GL_1(\End_A(M_A)),\]
which means $$\pi_{i}(\L_A(M_B), M_A) = \pi_{i-1}(GL_1(\End_A(M_A)))$$ for $i \geq 1$. In particular, one can give an explicit description of the homotopy groups of $\L_A(M_B)$:
\begin{itemize}
\item $\pi_0(\L_A(M_B))$ is the set of stable equivalence classes of $A$-module whose restriction to $\st(B)$ is $M_B$,
\item If $\L_A(M_B)$ is nonempty, then $\pi_1(\L_A(M_B),M_A) \cong Aut_{\st(A)}(M_A)$ where the $A$-module $M_A$ is chosen to be the basepoint of $\L_A(M_B)$, and,
\item for $i \geq 2$, $\pi_i(\L_A(M_B),M) \cong [M,M]^A_{i-1,0}.$ 
\end{itemize}

\begin{rk}
In the case when $M_B$ is the unit  $\un_B$ in $\st(B)$,  $\L_A(\un_B)$ clearly is nonempty as $\un_A$ is a point in $\L_A(S_B)$, and,  $$\pi_0(\L_A(S_B) \cong \Pic(A,B).$$ The higher homotopy groups are $$\pi_i(\L_A(\un_B),\un_A) \cong [\un_A,\un_A]^A_{i-1,0}.$$ 
So far we have described the negative homotopy groups in terms of $Ext$-groups in \eqref{de:grading}. For the unit, $\un_A$ one can use Poincar\'e duality (see \cite[Section 4]{Ric4}) to conclude   
\begin{equation} \label{Poincare}
\pi_i(\L_A(\un_B),\un_A) \cong \ext_A^{i-2,-|A|}(\un_A, \un_A)
\end{equation}
where $|A|$ is the maximum internal degree among elements in $A$.
Note that the latter is zero, for degree reasons.
\end{rk}

The forgetful functor $\mod(A) \rightarrow \mod(B)$ stabilizes in a strong symmetric monoidal functor
\begin{equation*}
U : \st(A) \rightarrow \st(B).
\end{equation*}
Furthermore, the functor $U$ has a right adjoint $F_B(A,-)$. Using this adjunction, one can produces an Endomorphism spectral sequence (see \cite{Ric4}), henceforth will be denoted by EndSS,
\begin{equation} \label{SS:EndSS}
_{\End}E_2^{s,n} \cong \bigoplus_{t} \ext_{(A\sur B)^*}^{n,t}(\un_{(A\sur B)^*}, \ext_B^{s-1,t-|B|}(S_B,S_B)) \Rightarrow \pi_{s-n}(\End_A(\un_A))
\end{equation}
which computes the homotopy groups of $\End_A(\un_A).$ 
\begin{rk}
Note that, we have the following chain of isomorphisms:
\begin{eqnarray*}
[\un_B,\un_B]_{s,t}^B & = & [U\un_A,\un_B]_{s,t}^{B} \\
& \cong & [\un_A, \hom_B(A,\un_B)]_{s,t}^{A} \\
& \cong & [\un_A, \un_A \otimes (A\sur B)^{\ast} ]_{s,t}^A.
\end{eqnarray*}
The action of $(A\sur B)^{\ast}$ is induced by the action of $(A\sur B)^{\ast}$ on itself in \linebreak $\ext^{s,t}_A(\un_A, (A\sur B)^{\ast} \otimes \un_A)$.
\end{rk}
EndSS \eqref{SS:EndSS} must be compared to the Cartan-Eilenberg spectral sequence \linebreak (see~\cite{Ra86}). The Cartan Eilenberg spectral sequence is a tri-graded spectral sequence 
\[ E_2^{n, s,t} =\bigoplus_{t'+t'' = t}\ext_{(A \sur B)^*}^{s,t'}(\un_{(A\sur B)^*}, \ext_{B}^{s', t''}(\un_B, \un_B)) \Rightarrow \ext_{A}^{n+s', t}(\un_A, \un_A). \]
which can be extended to compute the bigraded stable homotopy in the stable category 
\begin{equation} \label{SS:CESS}
 _{\mathit{cess}}E_2^{n,s,t} = \bigoplus_{t'+t'' = t} \ext_{(A \sur B)^*}^{n,t'}(\un_{(A\sur B)^*},[\un_B,\un_B]_{s,t''}^B) \Rightarrow [\un_A, \un_A]^A_{s-n, t}
\end{equation}
which we will refer to as CESS. Using the fact that $\pi_i(\End(\un_A)) = [\un_A, \un_A]_{i,0}$ along we the Poincar\'e duality isomorphism, it is easy to see that EndSS \eqref{SS:EndSS} is the restriction of CESS \eqref{SS:CESS} to $t=0$. 

Similar to the EndSS \eqref{SS:EndSS}, we have a spectral sequence 
\begin{equation*}
E_1^{s,n}(\L_A(M_B)) \Rightarrow \pi_{s-n}(\L_A(M_B)),
\end{equation*}
 where $s$ is a homological degree, and $n$ the spectral sequence degree. We call this spectral sequence Algebraic descent spectral sequence or Alg-DSS. The first differential is induced by the product on $(A\sur B)^{\ast}$. Thus we have (see \cite[Corollary 8.11]{Ric4})
\begin{equation} \label{SS:Lamb}
_{\mathit{dss}}E_2^{s,n} \cong \bigoplus_{t} \ext_{(A\sur B)^*}^{n,t}(\un_{(A\sur B)^*}, \ext_B^{s-2,t-|B|}(\un_B,\un_B)) \Rightarrow \pi_{s-n}(\L_A(\un_B)).
\end{equation}

Taking a minimal resolution of $\un_B$, one can assume that 
 $${}_{\mathit{cess}}E_1^{*,s,t} \cong \F[\theta] \otimes \ext_{B}^{s,t - * \cdot |\tau| }(\un_B,\un_B) $$ 
where $|\theta| = (1,0, |\tau|)$, since $\tau$ is the Bockstein associated to $\theta$. Note that the $E_2$-pages of spectral sequences of \eqref{SS:EndSS} and \eqref{SS:Lamb} are isomorphic up to a shift, hence we may assume that ${_{\mathit{dss}}E_1^{s,n}} \cong {_{\End}E_1^{s-1,n}}$. As a result we get the following lemma.
\begin{lemma} \label{lemma:obstructionline}
The elements in $_{\mathit{dss}}E_1^{s,n}$  with $s-n =0$, are of the form $\theta^n \otimes y$ 
where \[y \in \ext_B^{n-2,n|\tau|-|B|}(\un_B,\un_B).\] 
\end{lemma} 
Many but not all, higher differentials in Alg-DSS \eqref{SS:Lamb} can be imported  from the stable Cartan Eilenberg spectral sequence  \eqref{SS:EndSS} using the following comparison tool that has been established in \cite[Section 5]{MS14} (also see~\cite[Proposition 7.6]{Ric4}).
\begin{pro}[{\cite{MS14a},\cite[5.2.4, Remark 5.2.5]{MS14}}] \label{pro:MScomparison}
 The differential of length $r$ originating at $(s-1,0,n)$ in ${_{\End}E_r^{s-1,n}}$ coincides with the differential originating at ${_{\mathit{dss}}E_r^{s,n}}$, if $s \geq r+1$. 
\end{pro}

Note that if $\theta^q = 0$ in $\ext_{A}^{*,*}(\un_A, \un_A)$ then the $d_{q}$ is the last possible differential in CESS \eqref{SS:CESS} and hence in EndSS \eqref{SS:EndSS}. Since $\pi_{-1}(\End_A(\un_A))= 0$ means that all the elements in ${_{\End}E_{1}^{n+1,n}}$ are either zero or not present in the $E_{\infty}$-page of EndSS. Thus by Proposition~\ref{pro:MScomparison} all the in ${_{\mathit{dss}}E_{1}^{n,n}}$ for $n \geq q+1$ cannot be a nonzero permanent cycle. Hence the potential nonzero elements in $\pi_0(\L_A(M_B))$ are of the form 
\[ \theta^{n} \otimes y \in {_{\mathit{dss}}E_{1}^{n,n}}\]
where $n \leq q$ and $y$ has bidegree $(n-2, n|\tau| -B)$. Suppose $(q+1)|\tau| < |B|$, then $y$ is forced to be trivial whenever $n \leq q$. Thus we can conclude: 
\begin{lemma} \label{lem:false} If $B \subset A$ be a normal Hopf subalgebra of a connected Hopf algebra such that 
\begin{itemize}
\item $A \sur B^* = E(\tau)$,
\item $\theta^{q} = 0$ in $\ext_{A}^{*,*}(\un_A, \un_A)$, and, 
\item $(q+1)|\tau| < |B|$,
\end{itemize}
then $\Pic(A,B)= 0$.
\end{lemma}
\begin{thm} \label{thm:picachu}
The morphism of groups
\begin{equation*}
\can : \Z \oplus \Z \rightarrow \Pic(\A(2)),
\end{equation*}
which sends $(n,m)$ to $\un^{n,m}_{\A(2)}$, is an isomorphism.
\end{thm}
\begin{proof} When $A= C(2)$ and $B=D(2)$, then $A\sur B^* = E(\xi_1^2)$ and $\xi_1^2$ is Bockstein to $h_{11}$. Since $h_{11}^{3} = 0$  in $\ext_{C(2)}^{*,*}(\un_{C(2)}, \un_{C(2)})$ (see Proposition~\ref{pro:HC2}), we see that the all the criterias of Lemma~\ref{lem:false} is satisfied. Hence, $\Pic(C(2),D(2)) = 0$. Since we have already established $\Pic(D(2))=0$ (Proposition~\ref{pro:picD2}), it follows that $\Pic(C(2)) = 0$. This completes our first descent. 

For the second descent with $A= \A(2)$ and $B=C(2)$, $A\sur B^*= E(\xi_1^4)$ and $\xi_1^4$ is Bockstein to $h_{12}$ which satisfies the relation  $h_{12}^3=0$ in $\ext_{\A(2)}^{*,*}(\un_{\A(2)}, \un_{\A(2)})$ (see Proposition~\ref{pro:HA2}). Easy to check that all the criterias of Lemma~\ref{lem:false} is satisfied, therefore $\Pic(\A(2), C(2)) =0$, hence $\Pic(\A(2)) \cong \Z \times \Z$. 
\end{proof}



\newpage
\appendix
\section{Various extension groups}

The aim of this section is to describe the computations of  the bigraded $\ext$-groups of $\F$ over the Hopf algebras $D(2)$, $C(2)$ and $\A(2)$. The May spectral sequence is the most suitable tool for these sorts of computations which we briefly recall. It is convenient to work with the dual i.e. $D(2)^*$, $C(2)^*$ and $\A(2)^*$, for this purpose. 

May spectral sequence is obtained by assigning an additional filtration, called the \emph{May filtration}, to the Hopf algebra such that every element is primitive modulo the filtration. This filtration was introduced by J.P.May in his thesis \cite{May64}. Let the May filtration of $\xi_i$ is $2i-1$ following Ravenel~\cite{Ra86}. As a result we get a filtration on the cobar complexes for $D(2)^*$, $C(2)^*$ and $\A(2)^*$, which produces the trigraded May spectral sequence  
\[\begin{array}{lll}
E_1^{s,t,*} := \F[h_{10}, h_{20}, h_{21}, h_{30} ] &\Rightarrow& \ext_{C(2)}^{s,t}(\F, \F) \\
E_1^{s,t,*} :=\F[h_{10}, h_{11}, h_{20}, h_{21}, h_{30} ] &\Rightarrow& \ext_{C(2)}^{s,t}(\F, \F) \\
E_1^{s,t,*} :=  \F[h_{10}, h_{11}, h_{12}, h_{20}, h_{21}, h_{30} ] &\Rightarrow& \ext_{\A(2)}^{s,t}(\F, \F) 
\end{array} \] 
where  $h_{ij}$ corresponds to the class $[\xi_i^{2^j}]$ in the respective cobar complexes whose tridegree is $|h_{ij}| = (1,2^j(2^i-1),2i-1)$. The $d_1$-differentials in May SS comes from the coproducts structure and higher differentials can be computed using Nakamura's formula \cite{Nak72}, which is also described in \cite[Section~$2$]{BEM}. We display part of the $\ext$-groups in charts in $(t-s,s)$-coordinate system where each $\bullet$ represents an $\F$ vector space, vertical line represents multiplication by $h_{10}$, the slanted lines  of slope $\frac{1}{2}$ represent multiplication by $h_{11}$ and dotted lines of slope $\frac{1}{3}$ represent multiplication by $h_{12}$. 
\begin{pro} \label{pro:HD2}
The only non-trivial differentials in the May spectral sequence computing $\ext_{D(2)}(\F,\F)$ is $d_1(h_{30}) = h_{10}h_{21}$. Consequently,
\begin{equation*}
\ext^{s,t}_{D(2)}(\F,\F) \cong \frac{\F[h_{10},h_{20},h_{21},h_{30}^2]}{(h_{10}h_{21})}.
\end{equation*}
This $\F$-algebra is represented in figure \ref{fig:hD2}.
\end{pro}

\begin{figure}[h]
\begin{tikzpicture}[scale=0.6]

\clip (-1,-1) rectangle ( 20.50,  6.50);
\draw[color=lightgray] (0,0) grid [step=4]  (20,6);

\node [below] at (0,0) {$0$};
\node [below] at (4,0) {$4$};
\node [below] at (8,0) {$8$};
\node [below] at (12,0) {$12$};
\node [below] at (16,0) {$16$};
\node [below] at (20,0) {$20$};
\node [left] at (20.5,0) {$t-s$};
\node [left] at (0,6.2) {$s$};

\node [left] at (0.2,1) {$h_{10}$};
\node [left] at (2,1) {$h_{20}$};
\node [left] at (5,1) {$h_{21}$};

\node [left] at (0,0) {$0$};
\node [left] at (0,4) {$4$};

\draw [fill] ( 0.00, 0.00) circle [radius=0.05];
\draw [fill] ( 0.00, 1.00) circle [radius=0.05];
\draw [fill] ( 2.00, 1.00) circle [radius=0.05];
\draw [fill] ( 5.00, 1.00) circle [radius=0.05];
\draw [fill] ( 0.00, 2.00) circle [radius=0.05];
\draw [fill] ( 2.00, 2.00) circle [radius=0.05];
\draw [fill] ( 4.00, 2.00) circle [radius=0.05];
\draw [fill] ( 7.00, 2.00) circle [radius=0.05];
\draw [fill] (10.00, 2.00) circle [radius=0.05];
\draw [fill] (12.00, 2.00) circle [radius=0.05];
\draw [fill] ( 0.00, 3.00) circle [radius=0.05];
\draw [fill] ( 2.00, 3.00) circle [radius=0.05];
\draw [fill] ( 4.00, 3.00) circle [radius=0.05];
\draw [fill] ( 6.00, 3.00) circle [radius=0.05];
\draw [fill] ( 9.00, 3.00) circle [radius=0.05];
\draw [fill] (12.07, 2.93) circle [radius=0.05];
\draw [fill] (11.93, 3.07) circle [radius=0.05];
\draw [fill] (14.00, 3.00) circle [radius=0.05];
\draw [fill] (15.00, 3.00) circle [radius=0.05];
\draw [fill] (17.00, 3.00) circle [radius=0.05];
\draw [fill] ( 0.00, 4.00) circle [radius=0.05];
\draw [fill] ( 2.00, 4.00) circle [radius=0.05];
\draw [fill] ( 4.00, 4.00) circle [radius=0.05];
\draw [fill] ( 6.00, 4.00) circle [radius=0.05];
\draw [fill] ( 8.00, 4.00) circle [radius=0.05];
\draw [fill] (11.00, 4.00) circle [radius=0.05];
\draw [fill] (12.00, 4.00) circle [radius=0.05];
\draw [fill] (14.07, 3.93) circle [radius=0.05];
\draw [fill] (13.93, 4.07) circle [radius=0.05];
\draw [fill] (16.00, 4.00) circle [radius=0.05];
\draw [fill] (17.00, 4.00) circle [radius=0.05];
\draw [fill] (19.00, 4.00) circle [radius=0.05];
\draw [fill] (20.00, 4.00) circle [radius=0.05];
\draw [fill] ( 0.00, 5.00) circle [radius=0.05];
\draw [fill] ( 2.00, 5.00) circle [radius=0.05];
\draw [fill] ( 4.00, 5.00) circle [radius=0.05];
\draw [fill] ( 6.00, 5.00) circle [radius=0.05];
\draw [fill] ( 8.00, 5.00) circle [radius=0.05];
\draw [fill] (10.00, 5.00) circle [radius=0.05];
\draw [fill] (12.00, 5.00) circle [radius=0.05];
\draw [fill] (13.00, 5.00) circle [radius=0.05];
\draw [fill] (14.00, 5.00) circle [radius=0.05];
\draw [fill] (16.07, 4.93) circle [radius=0.05];
\draw [fill] (15.93, 5.07) circle [radius=0.05];
\draw [fill] (18.00, 5.00) circle [radius=0.05];
\draw [fill] (19.00, 5.00) circle [radius=0.05];
\draw [fill] ( 0.00, 6.00) circle [radius=0.05];
\draw [fill] ( 2.00, 6.00) circle [radius=0.05];
\draw [fill] ( 4.00, 6.00) circle [radius=0.05];
\draw [fill] ( 6.00, 6.00) circle [radius=0.05];
\draw [fill] ( 8.00, 6.00) circle [radius=0.05];
\draw [fill] (10.00, 6.00) circle [radius=0.05];
\draw [fill] (12.07, 5.93) circle [radius=0.05];
\draw [fill] (11.93, 6.07) circle [radius=0.05];
\draw [fill] (14.00, 6.00) circle [radius=0.05];
\draw [fill] (15.00, 6.00) circle [radius=0.05];
\draw [fill] (16.00, 6.00) circle [radius=0.05];
\draw [fill] (18.07, 5.93) circle [radius=0.05];
\draw [fill] (17.93, 6.07) circle [radius=0.05];
\draw [fill] (20.00, 6.00) circle [radius=0.05];

\draw ( 0.00, 1.00) --( 0.00, 0.00);
\draw ( 0.00, 2.00) --( 0.00, 1.00);
\draw ( 2.00, 2.00) --( 2.00, 1.00);
\draw ( 0.00, 3.00) --( 0.00, 2.00);
\draw ( 2.00, 3.00) --( 2.00, 2.00);
\draw ( 4.00, 3.00) --( 4.00, 2.00);
\draw (11.93, 3.07) --(12.00, 2.00);
\draw ( 0.00, 4.00) --( 0.00, 3.00);
\draw ( 2.00, 4.00) --( 2.00, 3.00);
\draw ( 4.00, 4.00) --( 4.00, 3.00);
\draw ( 6.00, 4.00) --( 6.00, 3.00);
\draw (12.00, 4.00) --(11.93, 3.07);
\draw (13.93, 4.07) --(14.00, 3.00);
\draw ( 0.00, 5.00) --( 0.00, 4.00);
\draw ( 2.00, 5.00) --( 2.00, 4.00);
\draw ( 4.00, 5.00) --( 4.00, 4.00);
\draw ( 6.00, 5.00) --( 6.00, 4.00);
\draw ( 8.00, 5.00) --( 8.00, 4.00);
\draw (12.00, 5.00) --(12.00, 4.00);
\draw (14.00, 5.00) --(13.93, 4.07);
\draw (15.93, 5.07) --(16.00, 4.00);
\draw ( 0.00, 6.00) --( 0.00, 5.00);
\draw ( 2.00, 6.00) --( 2.00, 5.00);
\draw ( 4.00, 6.00) --( 4.00, 5.00);
\draw ( 6.00, 6.00) --( 6.00, 5.00);
\draw ( 8.00, 6.00) --( 8.00, 5.00);
\draw (10.00, 6.00) --(10.00, 5.00);
\draw (11.93, 6.07) --(12.00, 5.00);
\draw (14.00, 6.00) --(14.00, 5.00);
\draw (16.00, 6.00) --(15.93, 5.07);
\draw (17.93, 6.07) --(18.00, 5.00);
\draw ( 0.00, 7.00) --( 0.00, 6.00);
\draw ( 2.00, 7.00) --( 2.00, 6.00);
\draw ( 4.00, 7.00) --( 4.00, 6.00);
\draw ( 6.00, 7.00) --( 6.00, 6.00);
\draw ( 8.00, 7.00) --( 8.00, 6.00);
\draw (10.00, 7.00) --(10.00, 6.00);
\draw (12.07, 6.93) --(12.07, 5.93);
\draw (11.93, 7.07) --(11.93, 6.07);
\draw (13.93, 7.07) --(14.00, 6.00);
\draw (16.00, 7.00) --(16.00, 6.00);
\draw (18.00, 7.00) --(17.93, 6.07);
\draw (19.93, 7.07) --(20.00, 6.00);

\end{tikzpicture}

\caption{The algebra $\ext_{D(2)}^{s,t}(\F,\F)$. An element in degree $(s,t)$ is plotted in $(s,t-s)$.} \label{fig:hD2}
\end{figure}
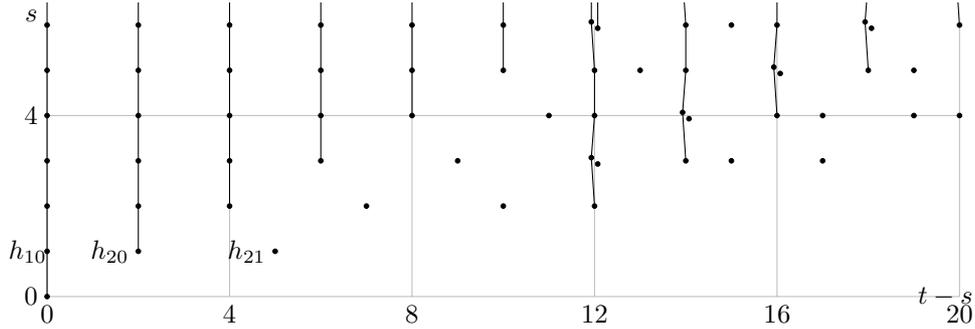

\begin{pro} \label{pro:HC2}
The only non-trivial differentials in the May spectral sequence computing $\ext_{C(2)}^{s,t}(\F,\F)$ are
\begin{itemize}
\item $d_1(h_{20}) = h_{10}h_{11}, $
\item $d_1(h_{30}) = h_{10}h_{21}, $
\item $d_2(h_{30}^2) = h_{11}h_{21}^2$, and,
\item $d_2(h_{20}^2) = h_{11}^3. $
\end{itemize}
This $\F$-algebra is represented in figure \ref{fig:hC2}.
\end{pro}

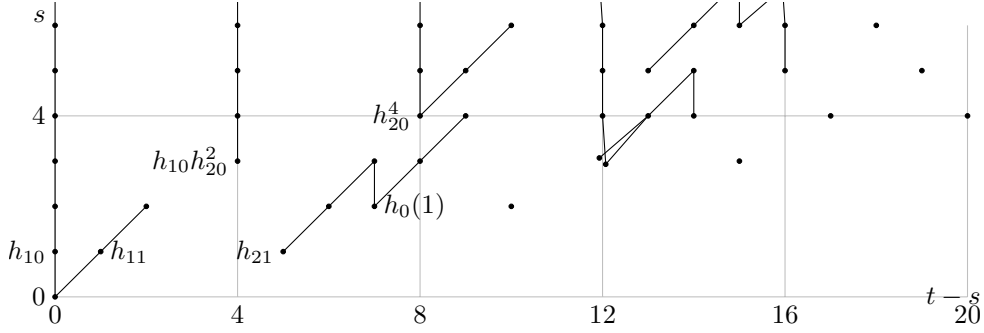
\begin{figure}

\begin{tikzpicture}[scale=0.6]

\clip (-1,-1) rectangle ( 20.50,  6.50);
\draw[color=lightgray] (0,0) grid [step=4]  (20,6);

\node [left] at (0,1) {$h_{10}$};
\node [right] at (1,1) {$h_{11}$};
\node [left] at (4,3) {$h_{10}h_{20}^2$};
\node [left] at (8,4) {$h_{20}^4$};
\node [right] at (7,2) {$h_0(1)$};
\node [left] at (5,1) {$h_{21}$};
\node [left] at (20.5,0) {$t-s$};
\node [left] at (0,6.2) {$s$};

\node [below] at (0,0) {$0$};
\node [below] at (4,0) {$4$};
\node [below] at (8,0) {$8$};
\node [below] at (12,0) {$12$};
\node [below] at (16,0) {$16$};
\node [below] at (20,0) {$20$};

\node [left] at (0,0) {$0$};
\node [left] at (0,4) {$4$};

\draw [fill] ( 0.00, 0.00) circle [radius=0.05];
\draw [fill] ( 0.00, 1.00) circle [radius=0.05];
\draw [fill] ( 1.00, 1.00) circle [radius=0.05];
\draw [fill] ( 5.00, 1.00) circle [radius=0.05];
\draw [fill] ( 0.00, 2.00) circle [radius=0.05];
\draw [fill] ( 2.00, 2.00) circle [radius=0.05];
\draw [fill] ( 6.00, 2.00) circle [radius=0.05];
\draw [fill] ( 7.00, 2.00) circle [radius=0.05];
\draw [fill] (10.00, 2.00) circle [radius=0.05];
\draw [fill] ( 0.00, 3.00) circle [radius=0.05];
\draw [fill] ( 4.00, 3.00) circle [radius=0.05];
\draw [fill] ( 7.00, 3.00) circle [radius=0.05];
\draw [fill] ( 8.00, 3.00) circle [radius=0.05];
\draw [fill] (12.07, 2.93) circle [radius=0.05];
\draw [fill] (11.93, 3.07) circle [radius=0.05];
\draw [fill] (15.00, 3.00) circle [radius=0.05];
\draw [fill] ( 0.00, 4.00) circle [radius=0.05];
\draw [fill] ( 4.00, 4.00) circle [radius=0.05];
\draw [fill] ( 8.00, 4.00) circle [radius=0.05];
\draw [fill] ( 9.00, 4.00) circle [radius=0.05];
\draw [fill] (12.00, 4.00) circle [radius=0.05];
\draw [fill] (13.00, 4.00) circle [radius=0.05];
\draw [fill] (14.00, 4.00) circle [radius=0.05];
\draw [fill] (17.00, 4.00) circle [radius=0.05];
\draw [fill] (20.00, 4.00) circle [radius=0.05];
\draw [fill] ( 0.00, 5.00) circle [radius=0.05];
\draw [fill] ( 4.00, 5.00) circle [radius=0.05];
\draw [fill] ( 8.00, 5.00) circle [radius=0.05];
\draw [fill] ( 9.00, 5.00) circle [radius=0.05];
\draw [fill] (12.00, 5.00) circle [radius=0.05];
\draw [fill] (13.00, 5.00) circle [radius=0.05];
\draw [fill] (14.00, 5.00) circle [radius=0.05];
\draw [fill] (16.00, 5.00) circle [radius=0.05];
\draw [fill] (19.00, 5.00) circle [radius=0.05];
\draw [fill] ( 0.00, 6.00) circle [radius=0.05];
\draw [fill] ( 4.00, 6.00) circle [radius=0.05];
\draw [fill] ( 8.00, 6.00) circle [radius=0.05];
\draw [fill] (10.00, 6.00) circle [radius=0.05];
\draw [fill] (12.00, 6.00) circle [radius=0.05];
\draw [fill] (14.00, 6.00) circle [radius=0.05];
\draw [fill] (15.00, 6.00) circle [radius=0.05];
\draw [fill] (16.00, 6.00) circle [radius=0.05];
\draw [fill] (18.00, 6.00) circle [radius=0.05];

\draw ( 0.00, 1.00) --( 0.00, 0.00);
\draw ( 1.00, 1.00) --( 0.00, 0.00);
\draw ( 0.00, 2.00) --( 0.00, 1.00);
\draw ( 2.00, 2.00) --( 1.00, 1.00);
\draw ( 6.00, 2.00) --( 5.00, 1.00);
\draw ( 0.00, 3.00) --( 0.00, 2.00);
\draw ( 7.00, 3.00) --( 6.00, 2.00);
\draw ( 7.00, 3.00) --( 7.00, 2.00);
\draw ( 8.00, 3.00) --( 7.00, 2.00);
\draw ( 0.00, 4.00) --( 0.00, 3.00);
\draw ( 4.00, 4.00) --( 4.00, 3.00);
\draw ( 9.00, 4.00) --( 8.00, 3.00);
\draw (12.00, 4.00) --(12.07, 2.93);
\draw (13.00, 4.00) --(12.07, 2.93);
\draw (13.00, 4.00) --(11.93, 3.07);
\draw ( 0.00, 5.00) --( 0.00, 4.00);
\draw ( 4.00, 5.00) --( 4.00, 4.00);
\draw ( 8.00, 5.00) --( 8.00, 4.00);
\draw ( 9.00, 5.00) --( 8.00, 4.00);
\draw (12.00, 5.00) --(12.00, 4.00);
\draw (14.00, 5.00) --(13.00, 4.00);
\draw (14.00, 5.00) --(14.00, 4.00);
\draw ( 0.00, 6.00) --( 0.00, 5.00);
\draw ( 4.00, 6.00) --( 4.00, 5.00);
\draw ( 8.00, 6.00) --( 8.00, 5.00);
\draw (10.00, 6.00) --( 9.00, 5.00);
\draw (12.00, 6.00) --(12.00, 5.00);
\draw (14.00, 6.00) --(13.00, 5.00);
\draw (16.00, 6.00) --(16.00, 5.00);
\draw ( 0.00, 7.00) --( 0.00, 6.00);
\draw ( 4.00, 7.00) --( 4.00, 6.00);
\draw ( 8.00, 7.00) --( 8.00, 6.00);
\draw (11.93, 7.07) --(12.00, 6.00);
\draw (15.00, 7.00) --(14.00, 6.00);
\draw (15.00, 7.00) --(15.00, 6.00);
\draw (16.07, 6.93) --(15.00, 6.00);
\draw (15.93, 7.07) --(16.00, 6.00);

\end{tikzpicture}
\caption{The algebra $\ext_{C(2)}^{s,t}(\F,\F)$. An element in degree $(s,t)$ is plotted in $(s,t-s)$. The element denoted $h_0(1)$ is $h_{20}h_{21} + h_{11}h_{30}$.} \label{fig:hC2}
\end{figure}

\begin{pro} \label{pro:HA2}
The only non-trivial differentials in the May spectral sequence computing $\ext_{\A(2)}^{s,t}(\F,\F)$ are
\begin{itemize}
\item $d_1(h_{20}) = h_{10}h_{11}, $
\item $d_1(h_{30}) = h_{10}h_{21} + h_{20}h_{12}, $
\item $d_1(h_{21}) = h_{11}h_{12},$
\item $d_2(h_{20}^2) = h_{11}^3+h_{10}^2h_{12}, $
\item $d_2(h_{21}^2) = h_{12}^3, $
\item $d_2(h_{30}^2) = h_{11} h_{21}^2$,
\item $d_2(h_{20}h_{21} + h_{11}h_{30}) = h_{10}h_{12}^2$, and, 
\item $d_4(h_{30}^4)= h_{12} h_{21}^4$.
\end{itemize}
This $\F$-algebra is represented in figure \ref{fig:HA2}.
\end{pro}

\begin{rk}
Note that the differential  $d_2(h_{20}h_{21} + h_{11}h_{30}) = h_{10}h_{12}^2$ is not a straightforward consequence of Nakamura's operations. The interested reader is reffered to \cite[Proposition 4.2]{Tan70} for the computation. One can compare the previous result to the list of relations in $\ext^{*,*}_{\A(2)}(\F,\F)$ given in \cite{SI67}.
\end{rk}

%
%
%
%

\begin{figure}[h]

\begin{tikzpicture}[scale=0.6]

\clip (-1,-1) rectangle ( 20.50,  6.50);
\draw[color=lightgray] (0,0) grid [step=4]  (20,6);

\node [left] at (20.5,0) {$t-s$};
\node [left] at (0,6.2) {$s$};
\node [left] at (0.2,1) {$h_{10}$};
\node [right] at (1,1) {$h_{11}$};
\node [right] at (3,1) {$h_{12}$};
\node [left] at (8,4) {$h_{20}^4$};
\node [right] at (8,3) {$h_{11}h_0(1)$};

\node [below] at (0,0) {$0$};
\node [below] at (4,0) {$4$};
\node [below] at (8,0) {$8$};
\node [below] at (12,0) {$12$};
\node [below] at (16,0) {$16$};
\node [below] at (20,0) {$20$};

\node [left] at (0,0) {$0$};
\node [left] at (0,4) {$4$};

\draw [fill] ( 0.00, 0.00) circle [radius=0.05];
\draw [fill] ( 0.00, 1.00) circle [radius=0.05];
\draw [fill] ( 1.00, 1.00) circle [radius=0.05];
\draw [fill] ( 3.00, 1.00) circle [radius=0.05];
\draw [fill] ( 0.00, 2.00) circle [radius=0.05];
\draw [fill] ( 2.00, 2.00) circle [radius=0.05];
\draw [fill] ( 3.00, 2.00) circle [radius=0.05];
\draw [fill] ( 6.00, 2.00) circle [radius=0.05];
\draw [fill] ( 0.00, 3.00) circle [radius=0.05];
\draw [fill] ( 3.00, 3.00) circle [radius=0.05];
\draw [fill] ( 8.00, 3.00) circle [radius=0.05];
\draw [fill] (12.00, 3.00) circle [radius=0.05];
\draw [fill] (15.00, 3.00) circle [radius=0.05];
\draw [fill] ( 0.00, 4.00) circle [radius=0.05];
\draw [fill] ( 8.00, 4.00) circle [radius=0.05];
\draw [fill] ( 9.00, 4.00) circle [radius=0.05];
\draw [fill] (12.00, 4.00) circle [radius=0.05];
\draw [fill] (14.00, 4.00) circle [radius=0.05];
\draw [fill] (15.00, 4.00) circle [radius=0.05];
\draw [fill] (17.00, 4.00) circle [radius=0.05];
\draw [fill] (18.00, 4.00) circle [radius=0.05];
\draw [fill] (20.00, 4.00) circle [radius=0.05];
\draw [fill] ( 0.00, 5.00) circle [radius=0.05];
\draw [fill] ( 8.00, 5.00) circle [radius=0.05];
\draw [fill] ( 9.00, 5.00) circle [radius=0.05];
\draw [fill] (11.00, 5.00) circle [radius=0.05];
\draw [fill] (12.00, 5.00) circle [radius=0.05];
\draw [fill] (14.00, 5.00) circle [radius=0.05];
\draw [fill] (15.00, 5.00) circle [radius=0.05];
\draw [fill] (17.00, 5.00) circle [radius=0.05];
\draw [fill] (18.00, 5.00) circle [radius=0.05];
\draw [fill] (20.00, 5.00) circle [radius=0.05];
\draw [fill] ( 0.00, 6.00) circle [radius=0.05];
\draw [fill] ( 8.00, 6.00) circle [radius=0.05];
\draw [fill] (10.00, 6.00) circle [radius=0.05];
\draw [fill] (11.00, 6.00) circle [radius=0.05];
\draw [fill] (12.00, 6.00) circle [radius=0.05];
\draw [fill] (14.00, 6.00) circle [radius=0.05];
\draw [fill] (17.00, 6.00) circle [radius=0.05];
\draw [fill] (20.00, 6.00) circle [radius=0.05];

\draw ( 0.00, 1.00) --( 0.00, 0.00);
\draw ( 1.00, 1.00) --( 0.00, 0.00);
\draw [dashed]  ( 3.00, 1.00) --( 0.00, 0.00);
\draw ( 0.00, 2.00) --( 0.00, 1.00);
\draw ( 2.00, 2.00) --( 1.00, 1.00);
\draw [dashed]  ( 3.00, 2.00) --( 0.00, 1.00);
\draw ( 3.00, 2.00) --( 3.00, 1.00);
\draw [dashed]  ( 6.00, 2.00) --( 3.00, 1.00);
\draw ( 0.00, 3.00) --( 0.00, 2.00);
\draw [dashed]  ( 3.00, 3.00) --( 0.00, 2.00);
\draw ( 3.00, 3.00) --( 2.00, 2.00);
\draw ( 3.00, 3.00) --( 3.00, 2.00);
\draw ( 0.00, 4.00) --( 0.00, 3.00);
\draw ( 9.00, 4.00) --( 8.00, 3.00);
\draw (12.00, 4.00) --(12.00, 3.00);
\draw [dashed]  (15.00, 4.00) --(12.00, 3.00);
\draw (15.00, 4.00) --(15.00, 3.00);
\draw [dashed]  (18.00, 4.00) --(15.00, 3.00);
\draw ( 0.00, 5.00) --( 0.00, 4.00);
\draw ( 8.00, 5.00) --( 8.00, 4.00);
\draw ( 9.00, 5.00) --( 8.00, 4.00);
\draw [dashed]  (11.00, 5.00) --( 8.00, 4.00);
\draw (12.00, 5.00) --(12.00, 4.00);
\draw (14.00, 5.00) --(14.00, 4.00);
\draw [dashed]  (15.00, 5.00) --(12.00, 4.00);
\draw (15.00, 5.00) --(14.00, 4.00);
\draw (15.00, 5.00) --(15.00, 4.00);
\draw [dashed]  (17.00, 5.00) --(14.00, 4.00);
\draw (17.00, 5.00) --(17.00, 4.00);
\draw [dashed]  (18.00, 5.00) --(15.00, 4.00);
\draw (18.00, 5.00) --(17.00, 4.00);
\draw (18.00, 5.00) --(18.00, 4.00);
\draw [dashed]  (20.00, 5.00) --(17.00, 4.00);
\draw (20.00, 5.00) --(20.00, 4.00);
\draw [dashed]  (21.00, 5.00) --(18.00, 4.00);
\draw (21.00, 5.00) --(20.00, 4.00);
\draw ( 0.00, 6.00) --( 0.00, 5.00);
\draw ( 8.00, 6.00) --( 8.00, 5.00);
\draw (10.00, 6.00) --( 9.00, 5.00);
\draw [dashed]  (11.00, 6.00) --( 8.00, 5.00);
\draw (11.00, 6.00) --(11.00, 5.00);
\draw (12.00, 6.00) --(12.00, 5.00);
\draw [dashed]  (14.00, 6.00) --(11.00, 5.00);
\draw (14.00, 6.00) --(14.00, 5.00);
\draw [dashed]  (17.00, 6.00) --(14.00, 5.00);
\draw (17.00, 6.00) --(17.00, 5.00);
\draw [dashed]  (20.00, 6.00) --(17.00, 5.00);
\draw (20.00, 6.00) --(20.00, 5.00);
\draw ( 0.00, 7.00) --( 0.00, 6.00);
\draw ( 8.00, 7.00) --( 8.00, 6.00);
\draw [dashed]  (11.00, 7.00) --( 8.00, 6.00);
\draw (11.00, 7.00) --(10.00, 6.00);
\draw (11.00, 7.00) --(11.00, 6.00);
\draw (12.00, 7.00) --(12.00, 6.00);

\end{tikzpicture}

\caption{The algebra $\ext_{\A(2)}^{s,t}(\F,\F)$. An element in degree $(s,t)$ is plotted in $(s,t-s)$.} \label{fig:HA2}
\end{figure}
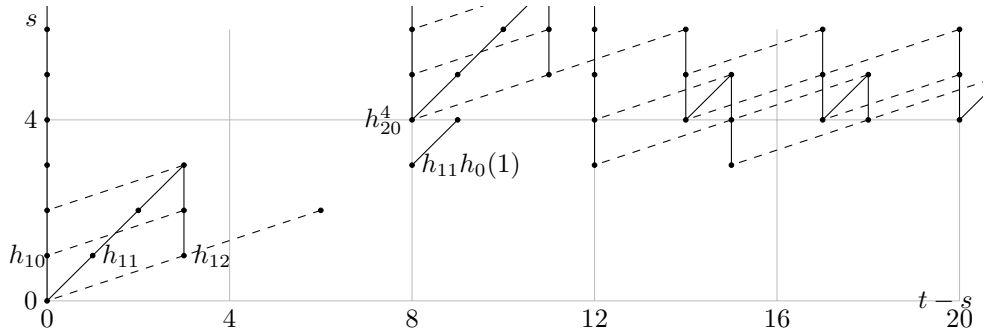

\bibliographystyle{alpha}
\bibliography{biblio}

\end{document}